\documentclass[11pt]{article}
\usepackage{amsmath, amssymb, amsthm}
\usepackage{verbatim}
\usepackage{multicol}
\usepackage{enumerate}
\usepackage{comment}
\usepackage[none]{hyphenat}
\usepackage{hyperref}
\hypersetup{
	colorlinks=true,
	linkcolor=blue,
	filecolor=magenta,      
	urlcolor=cyan,
	citecolor=blue
}

\date{}
\title{\vspace{-0.8cm}Tur\'an number of bipartite graphs with no $K_{t,t}$}
\author{Benny Sudakov\thanks{ETH Zurich, \emph{e-mail}: \textbf{\{benjamin.sudakov,istvan.tomon\}@math.ethz.ch}. Research supported by SNSF grant 200021-149111.}
	\and 
Istv\'an Tomon\footnotemark[1]
}

\theoremstyle{plain}
\newtheorem{theorem}{Theorem}

\newtheorem{claim}[theorem]{Claim}
\newtheorem{lemma}[theorem]{Lemma}
\newtheorem{conjecture}[theorem]{Conjecture}

\theoremstyle{definition}

\DeclareMathOperator{\ex}{ex}

\begin{document}
\maketitle

\begin{abstract}
  The extremal number of a graph $H$, denoted by $\ex(n,H)$, is the maximum number of edges in a graph on $n$ vertices that does not contain $H$. The celebrated K\H{o}v\'ari-S\'os-Tur\'an theorem says that
  for a complete bipartite graph with parts of size $t\leq s$ the extremal number is $\ex(K_{s,t})=O(n^{2-1/t})$. It is also known that this bound is sharp if $s>(t-1)!$. In this paper, we prove that if $H$ is a bipartite graph such that all vertices in one of its parts have degree at most $t$, but $H$ contains no copy of $K_{t,t}$, then $\ex(n,H)=o(n^{2-1/t})$. This verifies a conjecture of Conlon, Janzer and Lee.
\end{abstract}

\section{Introduction}
 Let $H$ be a graph. The extremal number of $H$, denoted by $\ex(n,H)$, is the maximum number of edges in a graph on $n$ vertices that does not contain $H$. By the classical Erd\H{o}s-Stone-Simonovits theorem \cite{ES66,ES46}, we have $\ex(n,H)=(1-\frac{1}{\chi(H)-1}+o(1))\binom{n}{2}$, where $\chi(H)$ is the chromatic number of $H$. Therefore, the order of $\ex(n,H)$ is known, unless $H$ is a bipartite graph. One of the major open problems in extremal graph theory is to understand the function $\ex(n,H)$ for bipartite graphs. The history of such results began in 1954 with the K\H{o}v\'ari-S\'os-Tur\'an theorem \cite{KST54}, which tells us that if $K_{s,t}$ is the complete bipartite graph with vertex classes of size $s\geq t$, then $\ex(n,K_{s,t})=O(n^{2-1/t})$. 
 This result was substantially extended by  F\"uredi \cite{F91} and Alon, Krivelevich, and Sudakov \cite{AKS03}.
 \begin{theorem}\label{thm:degreet}
  Let $H$ be a bipartite graph such that every vertex in one of its parts has degree at most $t$. Then $\ex(n,H)=O(n^{2-1/t})$.	
 \end{theorem}

It is known (see \cite{ARSz99,KRSz96}) that $\ex(n,K_{s,t})=\Theta(n^{2-1/t})$ if $s>(t-1)!$. Moreover, it is believed that $\ex(n,K_{t,t})=\Theta(n^{2-1/t})$ as well. 
This shows that in general if $H$ contains large complete bipartite subgraphs the above theorem is tight.
Thus, it is natural to ask what happens when the forbidden graph $H$ is $K_{t,t}$-free. For $t=2$ this question was considered by Erd\H{o}s \cite{E88} in 1988, who conjectured that if $H$ is a subgraph of a subdivision of another graph then there exists $\mu>0$ such that $\ex(n,H)=O(n^{3/2-\mu})$. A subdivison of a graph $\Gamma$ is obtained by replacing edges of $\Gamma$ by internally vertex disjoint paths of length two. By definition, if $H$ is a subgraph of a subdivision, then it is bipartite, has no $K_{2,2}$ and all the vertices in one of its parts have degree at most two. The conjecture of Erd\H{o}s was recently confirmed by Conlon and Lee \cite{CL19}, and in a stronger form by Janzer \cite{J19}. Conlon and Lee \cite{CL19} further proposed the following more general conjecture.

\begin{conjecture}\label{conj1}
For an integer $t\geq 2$, let $H$ be a $K_{t,t}$-free bipartite graph such that every vertex in one of the vertex classes of $H$ has  degree at most $t$. Then there is $\mu>0$ such that  $\ex(n,H)=O(n^{2-1/t-\mu})$.
\end{conjecture}

Despite recent progress on this topic (see, e.g., \cite{CL19, J19, CJL19, GJN19, JMY18, JQ19, KKL18}), this problem remains open for $t\geq 3$.  
Moreover the following weaker form of the above conjecture, proposed by Conlon, Janzer and Lee \cite{CJL19} was open as well. 

\begin{conjecture}\label{conj2}
Let $t\geq 2$ be an integer and let $H$ be a $K_{2,t}$-free bipartite graph such that every vertex in one of the parts of $H$ has  degree at most $t$. Then $\ex(n,H)=o(n^{2-1/t})$.
\end{conjecture}

Similar to Erd\H{o}s, one can also formulate this conjecture as a question on extremal numbers of subdivisions. For a hypergraph $\mathcal{H}$, the \emph{subdivision} of $\mathcal{H}$ is the bipartite graph $\mathcal{H}'$ whose two vertex classes are  $V(\mathcal{H})$ and $E(\mathcal{H})$, and $v\in V(\mathcal{H})$ and $e\in E(\mathcal{H})$ are joined by an edge if $v\in e$. Then Conjecture \ref{conj2} is equivalent to asking whether 
$\ex(n,\mathcal{H}')=o(n^{2-1/t})$ for a subdivision $\mathcal{H}'$ of a $t$-uniform hypergraph. In \cite{CJL19}, this conjecture is proved in the special case $\mathcal{H}$ is a linear hypergraph (that is, any two edges of $\mathcal{H}$ intersect in at most one vertex), which corresponds to the case in which the bipartite graph $H$ is $K_{2,2}$-free. Also, it is mentioned in \cite{CL19} and \cite{CJL19} that Conjecture \ref{conj2} holds in case $H$ is the subdivision of the complete $t$-uniform hypergraph with $t+1$ vertices, or the subdivision of a $t$-partite $t$-uniform hypergraph. 

In this paper we prove Conjecture \ref{conj2} in a very strong form, showing already that Conjecture \ref{conj1} holds with the same upper bound.

\begin{theorem}\label{thm:mainthm}
Let $t\geq 2$ be an integer. Let $H$ be a $K_{t,t}$-free bipartite graph such that every vertex in one of the parts of $H$ has degree at most $t$. Then $\ex(n,H)=o(n^{2-1/t})$.
\end{theorem}

\section{The extremal number of $K_{t,t}$-free bipartite graphs}

\subsection{Preliminaries}

In this section, we introduce our notation (which is mostly conventional), and state a few technical lemmas to prepare the proof of Theorem \ref{thm:mainthm}. We omit floors and ceilings whenever they are not crucial.

If $k$ is a positive integer and $X$ is a set, $X^{(k)}$ denotes the family of $k$ element subsets of $X$. If $G$ is a graph, $V(G)$ and $E(G)$ are the vertex set and edge set of $G$, respectively, and $v(G)=|V(G)|$, $e(G)=|E(G)|$. If $S\subset V(G)$, then $N_{G}(S)$ denotes the \emph{common neighborhood} of $S$, that is, the set of vertices that are joined to every element of $S$ by an edge. If $S=\{x\}$, we write simply $N_{G}(x)$ instead of $N_{G}(S)$. The degree of a vertex $x\in V(G)$ in $G$ is $d_{G}(x)=|N_{G}(x)|$. The complete $t$-uniform hypergraph on $k$ vertices is denoted by $K_{k}^{(t)}$.

In the proof of our main theorem, we use the following technical lemma, which can be found as Lemma 2.2 in \cite{CJL19}. Here, a graph $G$ is \emph{$K$-almost regular} if the maximum degree of $G$ is at most $K$-times the minimum degree. 

\begin{lemma}
	Let $c,\alpha>0$ such that $\alpha<1$. Let $n$ be a positive integer that is sufficiently large with respect to $c$ and $\alpha$. Let $G$ be a graph on $n$ vertices such that $e(G)\geq c n^{1+\alpha}$. Then $G$ contains a $K$-almost regular subgraph $G'$ on $m\geq n^{\frac{\alpha-\alpha^{2}}{4+4\alpha}}$ vertices such that $e(G')\geq \frac{2c}{5}m^{1+\alpha}$ and $K= 20\cdot 2^{\frac{1}{\epsilon^{2}}+1}$.
\end{lemma}

More precisely, we need the following immediate consequence of the above result.

\begin{lemma}\label{lemma:maxdeg}
	Let $0<c<10^{-4}$ and $\frac{1}{2}\leq \alpha<1$. Let $n$ be a positive integer that is sufficiently large with respect to $c$ and $\alpha$. Let $G$ be a graph on $n$ vertices such that $e(G)\geq c n^{1+\alpha}$. Then $G$ contains a bipartite subgraph $G'$, whose both vertex classes have size $m\geq \frac{1}{2}n^{\frac{\alpha-\alpha^{2}}{4+4\alpha}}$, $e(G')\geq \frac{c}{10}m^{1+\alpha}$ and the maximum degree of $G'$ is less than $m^{\alpha}$.
\end{lemma}

\begin{proof}
	By the previous lemma, $G$ contains a subgraph $G_{0}$ such that $G_{0}$ has $m_{0} \geq n^{\frac{\alpha-\alpha^{2}}{4+4\alpha}}$ vertices, $e(G_{0})\geq \frac{2c}{5}m_{0}^{1+\alpha}$ and $G_{0}$ is $K$-almost regular. 
	Since $\alpha>1/2$ we have $K<1000$. By randomly sampling the edges of $G_{0}$ with probability $p=\frac{2cm_{0}^{1+\alpha}}{5e(G_{0})}$ and using standard concentration arguments, we can find a subgraph $G_{0}'$ of  $G_{0}$ such that $G_{0}'$ is $2K$-almost regular and $\frac{4c}{5}m_{0}^{1+\alpha}\geq e(G_{0}')\geq \frac{c}{5}m_{0}^{1+\alpha}$.
	
	 But then the minimum degree of $G_{0}'$ is at most $\frac{8c}{5}m_{0}^{\alpha}$, so the maximum degree of $G_{0}'$ is less than $4Kcm_{0}^{\alpha}$. By well known folklore results, $V(G_{0}')$ can be partitioned into two sets $U$ and $V$ of size $m=\frac{1}{2}m_{0}$ such that the number of edges connecting $U$ and $V$ is at least $\frac{1}{2}e(G_{0}')\geq \frac{c}{10}m^{1+\alpha}$. Let $G'$ be the bipartite subgraph of $G_{0}'$ with vertex classes $U$ and $V$, then the maximum degree of $G'$ is less than $4Kcm_{0}^{\alpha}<8Kcm^{\alpha}<m^{\alpha}$. Therefore, $G'$ satisfies the desired conditions.     
\end{proof}

We will also use the hypergraph version of the classical Ramsey's theorem \cite{R30}.

\begin{lemma}\label{lemma:ramsey}
	Let $k,t$ be positive integers. Then there exists $\Delta=\Delta(k,t)$ such that any two coloring of the edges of the complete $t$-uniform hypergraph $K_{\Delta}^{(t)}$ contains a monochromatic copy of $K_{k}^{(t)}$.
\end{lemma}

Finally, we will use the celebrated Hypergraph Removal Lemma, proved independently by Nagle, R\"odl, Schacht \cite{NRS06} and Gowers \cite{G07}.  

\begin{lemma}\label{lemma:removal}
	Let $k,t$ be positive integers. For every $\beta>0$ there exists $\delta=\delta(k,t,\beta)>0$ such that the following holds. If $\mathcal{H}$ is a $t$-uniform hypergraph on $n$ vertices such that 
one needs to remove at least $\beta n^t$ edges of $\mathcal{H}$ to make it $K_{k}^{(t)}$-free, then $\mathcal{H}$
contains at least $\delta n^{k}$ copies of $K_{k}^{(t)}$.
\end{lemma}

\subsection{Overview of the proof}

Despite our proof being quite short, it might help to briefly outline the main ideas. 

Let $H_{k}$ be the bipartite graph with vertex classes $X$ and $Y$ such that $|X|=k$, $|Y|=(t-1)\binom{k}{t}$, and for every $t$-tuple  $S\in X^{(t)}$, there are exactly $t-1$ vertices in $Y$ whose neighborhood is equal to $S$. Clearly, 
for every $H$ there is a large enough integer $k$ such that $H$ is a subgraph of $H_{k}$. Therefore, it is enough to show that $\ex(n,H_{k})=o(n^{2-1/t})$. 

Let us fix an $H_{k}$-free graph $G$ on $n$ vertices with $\epsilon n^{2-1/t}$ edges, where we think of $\epsilon$ as a small constant. Then our goal is  to show that $n$ cannot be arbitrarily large. 
We first pass to a bipartite subgraph with parts $V$ and $W$, where $V$ is of order $n$, and $|W|$ is of order $n^{1-1/t}$. This is in contrast with few  previous papers in the same topic \cite{CJL19, CL19, J19} which work with a bipartite subgraph $G'$ of $G$ in which both parts have roughly the same size. By setting the parameters correctly, the advantage of our first step is that the average size of the common neighborhood in $V$ of the $(t-1)$-tuples of vertices from $W$ is some large constant. Next we consider the $t$-uniform hypergraph $\mathcal{H}$ on $W$ where each $e\in W^{(t)}$ is an edge if it has at least $t-1$ common neighbors in $V$. We color the edges of $\mathcal{H}$ by red and blue such that an edge is red if it has at least $(t-1)\binom{k}{t}$ common neighbors. One can argue that $\mathcal{H}$ cannot have a red  $K_{k}^{(t)}$, since otherwise we can find greedily a copy of $H_k$. Thus, using Ramsey's theorem, we find many blue copies of $K_{k}^{(t)}$. We further prove that one needs to remove many hyperedges to destroy all these blue copies of $K_{k}^{(t)}$. Therefore we can apply the Hypergraph Removal Lemma to show that $\mathcal{H}$ must contain $\Omega(|W|^{k})$ copies of $K_{k}^{(t)}$. Then by counting certain bad copies of $K_{k}^{(t)}$, we conclude that there must exist a copy $R$ such that the common neighborhoods $N_{G'}(S)$ for $S\in E(R)$ are all pairwise disjoint. Using such $R$ as one part of $H_{k}$ we can clearly embed the other part in $\bigcup_{S\in E(R)}N_{G'}(S)$.

\subsection{The proof of Theorem \ref{thm:mainthm}}

In this section, we present the proof of Theorem \ref{thm:mainthm}. Our proof works for all $t\geq 2$ but since the case $t=2$ is already known  by \cite{CL19,J19}, for computational convenience we assume that $t\geq 3$. 

Fix $k$ such that $H$ is contained in $H_{k}$, where $H_{k}$ is the bipartite graph defined in the previous section. We prove that for every $0<\epsilon<10^{-4}$ if $n$ is sufficiently large, then $\ex(n,H_{k})\leq \epsilon n^{2-1/t}$. Let $G$ be a graph with $n$ vertices and at least $\epsilon n^{2-1/t}$ edges, and assume that $G$ does not contain $H_{k}$. By Lemma \ref{lemma:maxdeg}, $G$ has a bipartite subgraph $G'$ with vertex classes $U$ and $V$,  $|U|=|V|=n'>\frac{1}{2}n^{\frac{(1-1/t)}{8t-4}}$ such that $e(G')\geq \frac{\epsilon}{10}(n')^{2-1/t}$, and the maximum degree of $G'$ is at most $(n')^{1-1/t}$. In the rest of the proof, we shall only work with $G'$ instead of $G$, so with slight abuse of notation, let $G:=G'$, $n:=n'$ and $\epsilon:=\frac{\epsilon}{10}$. Clearly, it is enough to prove that $G$ contains $H_{k}$ if $n$ is sufficiently large with respect to $k,t,\epsilon$. 

As the next step, we pass to an even smaller subgraph $G'$ of $G$ with parts of size roughly $n^{1-1/t}$ and $n$. 
This is done using the following claim.

\begin{claim}\label{claim:reduce}
	Let $n^{-1/t}<p<1$. If $n$ is sufficiently large with respect to $\epsilon$ and $t$, then here exists $W\subset U$ such that $\frac{pn}{2}<|W|<2pn$, the graph $G'=G[W\cup V]$ has at least $\frac{p}{4}e(G)$ edges, and $d_{G'}(x)<2pn^{1-1/t}$ holds for every $x\in V$.  
\end{claim}

\begin{proof}
 Pick each element of $U$ with probability $p$, and let $W$ be the set of selected vertices. Then the statement follows by standard concentration arguments. By Chernoff's inequality, with high probability we have $|d_{G'}(x)-pd_{G}(x)|<\frac{1}{2}pd_{G}(x)$ for every $x\in V$ satisfying $d_{G}(x)\geq n^{1/2}$. Also, with high probability, $||W|-pn|\leq \frac{1}{2}pn$. Therefore, there exists a choice for $W$ which satisfies these inequalities. But then every $x\in V$ satisfies $d_{G'}(x)<2pn^{1-1/t}$ as the maximum degree of $G'$ is at most $n^{1-1/t}$. Since $pe(G) \geq \Omega(n^{2-2/t}) \gg n^{3/2}$, we also have
 $$e(G')=\sum_{x\in V}d_{G'}(x)\geq \sum_{x\in V, d_{G}(x)\geq n^{1/2}} \frac{1}{2}pd_{G}(x) \geq \frac{1}{2}pe(G)-n \cdot n^{1/2} \geq \frac{1}{4}pe(G)\,.$$
\end{proof}

We would like to choose $p$ such that the average size of a common neighborhood of a $(t-1)$-tuple of vertices in $V$ is some large constant (independent of $n$). Let $n^{-1/t}<p<1$, which we will specify later, and let $W$ be a subset of $U$ satisfying the properties described in Claim \ref{claim:reduce}. Consider the sum $L=\sum_{C\in V^{(t-1)}}|N_{G'}(C)|.$ We have
\begin{align*}
L&=\sum_{x\in W}\binom{d_{G'}(x)}{t-1} \geq|W|\binom{e(G')/|W|}{t-1} >(t-1)^{-(t-1)}e(G')^{t-1}|W|^{-(t-2)}\\
&>(t-1)^{-(t-1)}\left(\frac{p\epsilon}{4}n^{2-1/t}\right)^{t-1}\left(2pn\right)^{-(t-2)}=\left(\frac{\epsilon}{t-1}\right)^{t-1}2^{-3t+4}pn^{t-1+1/t},
\end{align*}
where the first inequality holds by convexity.

By Ramsey's Theorem (Lemma \ref{lemma:ramsey}), there exists a positive integer $\Delta=\Delta(k,t)$ such that any red-blue coloring of the edges of the hypergraph $K_{\Delta}^{(t)}$  contains either a red or a blue copy of $K_{k}^{(t)}$. Choose $p$ such that $L\geq 2\Delta n^{t-1}$ holds. Then by the previous calculations, we can choose $p=\alpha n^{-1/t}$, where $\alpha=2\Delta (\frac{t-1}{\epsilon})^{t-1}2^{3t-4}$. The important thing to notice is that $\alpha=\alpha(k,t,\epsilon)$ does not depend on $n$. Also, we remark that $\frac{\alpha}{2}n^{1-1/t}< |W|< 2\alpha n^{1-1/t}$, and every $x\in V$ has degree at most $2pn^{1-1/t}=2\alpha n^{1-2/t}$ in $G'$.

Let $\mathcal{H}$ be the $t$-uniform hypergraph on $W$ in which $S\in W^{(t)}$ is an edge if $|N_{G'}(S)|\geq t-1$. Color an edge $S\in E(\mathcal{H})$ red if $|N_{G'}(S)|\geq (t-1)\binom{k}{t}$, and color $S$ blue, otherwise. If $\mathcal{H}$ contains a red clique of size $k$, then $G'$ contains $H_{k}$. Indeed, if $R\subset W$ spans a red clique of size $k$ in $\mathcal{H}$, then for each $S\in R^{(t)}$ one can greedily select a set $Q_{S}\subset N_{G'}(S)$ of $t-1$ vertices such that $Q_{S}$ and $Q_{S'}$ are disjoint if $S\neq S'$. 
This clearly gives a copy of $H_k$, contradiction. Therefore, in what comes, we can assume that $\mathcal{H}$ does not contain a red clique of size $k$.

Let $C\in V^{(t-1)}$ and consider $T=N_{G'}(C)$. Let $r=\lfloor \frac{|T|}{\Delta}\rfloor>\frac{|T|}{\Delta}-1$, and let $T_{1},\dots,T_{r}$ be disjoint sets of size $\Delta$ in $T$. Note that for $i=1,\dots,r$,  $\mathcal{H}[T_{i}]$ is a clique of size $\Delta$ in $\mathcal{H}$. But $\mathcal{H}[T_{i}]$ does not contain a red clique of size $k$, so by the definition of $\Delta$, $\mathcal{H}[T_{i}]$ contains a blue clique of size $k$, let $A_{i}$ be the vertex set of such a clique. Set $Z_C=\{A_{i}:i=1,\dots,r\}$, and let $Z$ be the multiset $\bigcup_{C\in V^{(t-1)}}Z_C$ (that is, we count each $k$-tuple with multiplicity $s$ if it appears in $s$ of the sets $Z_C$ for $C\in V^{(t-1)}$). Then $$|Z|= \sum_{C\in V^{(t-1)}}|Z_{C}|\geq \sum_{C\in V^{(t-1)}}\left(\frac{|N_{G'}(C)|}{\Delta}-1\right)=\frac{L}{\Delta}- \binom{n}{t-1}\geq n^{t-1}.$$

Next, we show that $Z$ contains a large subset in which the size of the intersection of any two elements is less than $t$. 

\begin{claim}
  There exists a constant $\beta=\beta(k,t,\epsilon)>0$ and $Z'\subset Z$ such that $|Z'|\geq \beta |W|^{t}$, and if $A,B\in Z'$, then $|A\cap B|<t$.
\end{claim}

\begin{proof}
Let $D$ be the graph on vertex set $Z$ in which $A$ and $B$ are joined by an edge if $|A\cap B|\geq t$. Let $A\in Z$ and let $S\in A^{(t)}$. Then $S$ is blue, so $|N_{G'}(S)|\leq (t-1)\binom{k}{t}=u$. But then there are at most $\binom{u}{t-1}$ sets  $C\in V^{(t-1)}$ such that $S\subset N_{G'}(C)$. For each such $C$, at most one element of $Z_C$ contains $S$, so in total at most $\binom{u}{t-1}$ elements of $Z$ contain $S$. Hence, as $A$ has $\binom{k}{t}$ subsets of size $t$, $A$ has degree at most $d=\binom{k}{t}\binom{u}{t-1}$ in $D$.

But then $D$ contains an independent set of size at least $\frac{|Z|}{d+1}\geq \frac{n^{t-1}}{d+1}>\beta|W|^{t}$, where $\beta=\frac{1}{(2\alpha)^t(d+1)}$. Let $Z'$ be such an independent set. 
\end{proof}

Note that by our construction, $Z'$ corresponds to a family of copies of $K_{k}^{(t)}$ in $\mathcal{H}$ such that no two copies share a hyperedge. Let $M$ be the total number of copies of $K_{k}^{(t)}$ in $\mathcal{H}$.

\begin{claim}\label{claim:copies}
	There exists a constant $\gamma=\gamma(k,t,\epsilon)$ such that $M\geq \gamma n^{(t-1)k/t}$.  
\end{claim}

\begin{proof}
In order to destroy every copy of $K_{k}^{(t)}$ in $\mathcal{H}$, one needs to remove at least one hyperedge from every element of $Z'$, which results in the removal of at least $\beta|W|^{t}$ edges. 
Let $\delta=\delta(k,t,\beta)$ be the constant given by the Hypergraph Removal Lemma (Lemma \ref{lemma:removal}).
Then $M\geq \delta |W|^{k}\geq \delta (\frac{\alpha}{2})^{k}n^{(t-1)k/t}$. Choosing $\gamma=\delta(\frac{\alpha}{2})^{k}$ completes the proof.
\end{proof}

Let us say that a copy $R$ of $K_{k}^{(t)}$ in $\mathcal{H}$ is \emph{bad} if there exists two distinct $t$-tuples $S,S'\in E(R)$ such that $N(S)\cap N(S')\neq \emptyset$, otherwise, say that $R$ is \emph{good}. Clearly, if there exists a good copy of $K_{k}^{(t)}$, then $G'$ contains $H_{k}$. Indeed, if $R$ is a good copy, then  for every $S\in R^{(t)}$, let $Q_{S}$ be any $(t-1)$-element subset of $N_{G'}(S)$, then the sets $Q_{S}$ for $S\in E(R)$ are pairwise disjoint, so there is copy of $H_{k}$ with vertex set $R\cup \bigcup_{S\in R^{(t)}}Q_{S}$. To show that there is a good copy of $K_{k}^{(t)}$, let us count the number of bad copies.
  
  \begin{claim}\label{claim:badcopies}
  	There exists a constant $\gamma'=\gamma'(k,t,\epsilon)$ such that the number of bad copies of $K_{k}^{(t)}$ is at most $\gamma'n^{(k(t-1)-1)/t}$.
  \end{claim}

  \begin{proof}
     If $R$ is a bad copy of $K_{k}^{(t)}$ in $\mathcal{H}$, then there are two sets $S,S'\in E(R)$ such that $N_{G'}(S)\cap N_{G'}(S')$ is non-empty. Let $x\in N_{G'}(S)\cap N_{G'}(S')$. Then $N_{G'}(x)$ contains $S\cup S'$. This implies that $|N_{G'}(x)\cap V(R)|\geq |S\cup S'|\geq t+1$.
      Therefore, summing over all the vertices $x\in V$ we have that
      the number of bad copies of $K_{k}^{(t)}$ is at most $$\sum_{x\in V} \binom{d_{G'}(x)}{t+1}|W|^{k-t-1}\leq n(2\alpha n^{1-2/t})^{t+1}(2\alpha n^{1-1/t})^{k-t-1}=(2\alpha)^{k}n^{(k(t-1)-1)/t}.$$
     Setting $\gamma'=(2\alpha)^{k}$ suffices.
  \end{proof}
To conclude the proof, note that if $n$ is sufficiently large as a function of  $k,t,\epsilon$, then by Claims \ref{claim:copies} and \ref{claim:badcopies} there is a good copy of $K_{k}^{(t)}$, implying $G'$ contains $H_{k}$, a contradiction.

\section{Concluding remarks}

Although Conjecture \ref{conj1} remains open, our proof of Theorem \ref{thm:mainthm} can be slightly modified to confirm the conjecture for the following general family of bipartite graphs $H$. If $\mathcal{H}$ is a hypergraph, define the \emph{$r$-fold subdivision} of $\mathcal{H}$, denoted by $\mathcal{H}^{[r]}$, as follows: let the two vertex classes of $\mathcal{H}^{[r]}$ be $V(\mathcal{H})$ and $E=E_{1}\cup\dots\cup E_{r}$, where $E_{1},\dots,E_{r}$ are disjoint copies of $E(\mathcal{H})$, and $e\in E_{i}$ is joined to $v \in V(\mathcal{H})$ by an edge if $v\in e$. Theorem \ref{thm:mainthm} is equivalent to the statement that if $\mathcal{H}$ is a $t$-uniform hypergraph, then $\ex(n,\mathcal{H}^{[t-1]})=o(n^{2-1/t})$. However, in case $\mathcal{H}$ is $t$-partite, we can do slightly better.

\begin{theorem}
	Let $\mathcal{H}$ be a $t$-partite $t$-uniform hypergraph. Then there exists $\mu>0$ such that $$\ex(n,\mathcal{H}^{[t-1]})=O(n^{2-1/t-\mu}).$$
\end{theorem}

To prove this theorem, one can use the proof of our main result. The only difference is in the application of the Hypergraph Removal Lemma. When we need to count the copies of a $t$-partite hypergraph instead of $K_{k}^{(t)}$, we can choose  $\delta$ to be a polynomial of $\beta$ (whose degree depends only on the $t$-partite hypergraph in question). This fact follows from an approach used by Erd\H{o}s \cite{E64} to bound extremal numbers of complete $t$-uniform $t$-partite hypergraphs (for an application of this technique for counting copies of such hypergraphs see, e.g.,  
Proposition 3.6 in \cite{CFS16}). Therefore, one can take $\epsilon=n^{-\mu}$ with some small enough $\mu$ in order for our arguments to work. We omit further details.

\end{document}